\documentclass{article}
\usepackage{amsmath, amssymb, amsthm}
\usepackage{graphicx}
\usepackage{epstopdf}
\usepackage{caption}
\usepackage{subcaption}
\usepackage{color}
\usepackage{fullpage}
\usepackage{hyperref}
\usepackage{tabularx}
\usepackage{enumerate}
\usepackage[vcentermath]{youngtab}
\usepackage[utf8]{inputenc}
\usepackage{young}

\newcommand{\Bea}{\begin{eqnarray*}}
\newcommand{\Eea}{\end{eqnarray*}}
\newcommand{\bea}{\begin{eqnarray}}
\newcommand{\eea}{\end{eqnarray}}

\newcommand{\eq}[1]{\begin{align}#1\end{align}}

\newcommand{\f}{\frac}

\newcommand{\C}{\mathbb{C}}
\newcommand{\E}{\mathbb{E}}
\newcommand{\N}{\mathbb{N}}
\newcommand{\R}{\mathbb{R}}

\newcommand{\Pb}{\mathbb{P}}

\newcommand{\CF}{\mathcal{F}}

\newcommand{\comment}[1]{}

\def\g{\gamma}
\def\d{\delta}
\def\s{\sigma}
\def\l{\lambda}

\def\1{\textbf{1}}

\def\tr{\text{tr}}

\newtheorem{thm}{Theorem}[section]
\newtheorem{prop}[thm]{Proposition}
\newtheorem{lem}[thm]{Lemma}

\theoremstyle{definition}
\newtheorem{dfn}[thm]{Definition}
\newtheorem{rem}[thm]{Remark}

\allowdisplaybreaks

\title{Distribution of singular values of random band matrices; Marchenko-Pastur law and more}
\author{Indrajit Jana\footnote{Department of Mathematics, University of California, Davis; ijana@math.ucdavis.edu}, Alexander Soshnikov\footnote{Department of Mathematics, University of California, Davis; soshniko@math.ucdavis.edu}}
\date{\today}

\begin{document}
\maketitle

\begin{abstract}
We consider the limiting spectral distribution of matrices of the form $\f{1}{2b_{n}+1} (R + X)(R + X)^{*}$, where $X$ is an $n\times n$ band
matrix of bandwidth $b_{n}$ and $R$ is a non random band matrix of bandwidth $b_{n}$. We show that the Stieltjes transform
of ESD of such matrices converges to the Stieltjes transform of a non-random measure. And the limiting Stieltjes
transform satisfies an integral equation. For $R=0$, the integral equation yields the Stieltjes transform of the Marchenko-Pastur
law.
\end{abstract}


\section{Introduction}  

Random matrices play a crucial role in several scientific research including Nuclear Physics, Signal Processing, Numerical linear algebra etc. In 1950s',
Wigner studied Random Band Matrices (RBM) in the context of Nuclear Physics ~\cite{wigner1955}. Tridiagonal RBM can be used to  approximate
random Schr\"odinger operator. RBM can also be used to model a particle system where interactions are stronger for nearby particles.
Casati et al. studied RBM in the context of quantum chaos ~\cite{casati1990scaling}. A study of RBM in the framework of supersymmetric approach
can be found
in ~\cite{fyodorov1991scaling}. Properties of RBM with strongly fluctuating diagonal entries and sparse RBM were studied by Fyodorov, Mirlin,
and co-authors~\cite{fyodorov1995statistical}, ~\cite{fyodorov1996wigner}. In addition, RBM appear in the studies of conductance fluctuations of
quasi-one dimesnional disordered systems ~\cite{devillard1991statistics}, the kicked quantum rotator ~\cite{scharf1989kicked},
systems of interacting particles in a random potential ~\cite{shepelyansky1994coherent, jacquod1995hidden}.

In this paper, we consider random band matrices
of the form $\f{1}{2b_{n}+1}(R+X)(R+X)^{*}$, where $X$ is an $n\times n$ band matrix of bandwidth $b_{n}$ with iid entries and $R$ is a nonrandom
band matrix.
We study the limiting empirical distribution of the eigenvalues of such matrices.

 Let $M_{n}$ be an $n\times n$ matrix. Let $\l_{1},\ldots,\l_{n}$ be the eigenvalues of $M_{n}$ and
\Bea
\mu_{n}(x,y):=\f{1}{n}\#\{\l_{i},1\leq i\leq n:\Re(\l_{i})\leq x,\Im(\l_{i})\leq y\}
\Eea
be the empirical spectral distribution (ESD) of $M$. Ginibre \cite{ginibre1965statistical} showed that if $M_{n}=\f{1}{\sqrt{n}}X_{n}$, where $x_{ij}$, the entries of $X_{n}$, are iid complex normal variables, then the joint density of $\l_{1},\ldots,\l_{n}$ is given by
\Bea
f(\l_{1},\ldots,\l_{n})=c_{n}\prod_{i<j}|\l_{i}-\l_{j}|^{2}\prod_{i=1}^{n}e^{-n|\l_{i}|^{2}},
\Eea
where $c_{n}$ is the normalizing constant. Using this, Mehta \cite{mehta1967random} showed that $\mu_{n}$ converges to the uniform distribution on the unit disk. Later on Girko \cite{girko1985circular} and Bai \cite{bai1997circular} proved the result under more relaxed assumptions, namely under the assumption that $\E|X_{ij}|^{6}<\infty$. Proving the result only under second moment assumption was open until Tao and Vu \cite{tao2008random, tao2010random}.  

Following the method used by Girko, and Bai, the real part of the Stieltjes transform $m_{n}(z):=\f{1}{n}\sum_{i=1}^{n}\f{1}{\l_{i}-z}$ can be written as
\Bea
m_{nr}(z)&:=&\Re(m_{n}(z))\\
&=&\f{1}{n}\sum_{i=1}^{n}\f{\Re(\l_{i}-z)}{|\l_{i}-z|^{2}}\\
&=&-\f{1}{2}\f{\partial}{\partial\Re(z)}\int_{0}^{\infty}\log x\nu_{n}(dx,z),
\Eea
where $\nu_{n}(\cdot,z)$ is the ESD of $(\f{1}{\sqrt{n}}X_{n}-zI)(\f{1}{\sqrt{n}}X_{n}-zI)^{*}$, and $z\in \C^{+}:=\{z\in \C:\Im(z)>0\}$. And secondly the characteristic function of $\f{1}{\sqrt{n}}X_{n}$ satisfies \cite[section 1]{girko1985circular}
\Bea
\int\int e^{i(ux+vy)}\mu_{n}(dx,dy)=\f{u^{2}+v^{2}}{i4\pi u}\int\int\f{\partial}{\partial s}\left[\int_{0}^{\infty}(\log x)\;\nu_{n}(dx,z)\right]e^{i(us+vt)}\;dtds,
\Eea
for any $uv\neq 0$, and where $z=s+it$. 

So, finding the limiting behaviour of $\nu_{n}(\cdot,z)$ is an essential ingredient in finding the limiting behaviour of $\mu_{n}(\cdot,\cdot)$. However, as described in \cite{tao2010random}, a good estimate of the smallest singular value of random matrix is needed to prove the Circular law. Finding a estimate of the smallest singular value is not a part of this paper. In this article, we will focus on finding the limiting behaviour of $\nu_{n}(\cdot,z)$ for RBM so that it can be used for finding the limiting behaviour of $\mu_{n}(\cdot,\cdot)$ for RBM. 

We consider the limiting ESD of matrices of the form $\f{1}{2b_{n}+1}(R+X)(R+X)^{*}$, where $X$ is an $n\times n$ band matrix of bandwidth $b_{n}$ and $R$ is a non RBM. Silverstein, Bai, and Dozier considered the ESD of $\f{1}{n}(R+X)(R+X)^{*}$ type of matrices where $X$ was $m\times n$ rectangular matrix with iid entries, $R$ was a matrix independent of $X$, and the ratio $\f{m}{n}\to c\in (0,\infty)$   \cite{silverstein1995strong, silverstein1995empirical, dozier2007empirical}. Having the same bandwidth for $R$ and $X$ simplifies the calculation. But we do not think that we need the same bandwidth. Thanks to the referees for pointing this out.

This paper is organized in the following way; in the section \ref{Chapter Band_ESD:section: Main results}, we formulate the band matrix model and state the main results. In section \ref{Cahpter Band_ESD:section: main proof of the theorem}, we give the main idea of the proof. In section \ref{Chapter Band_ESD:Section: Main concentration results}, we prove two concentration results which are the main ingredients of the proof. And in the section \ref{Chapter Band_ESD:section: Appendix}, we provide some tools and the proofs for interested readers.

\section{Main Results}\label{Chapter Band_ESD:section: Main results}
\begin{dfn}[Periodic band matrix] An $n\times n$ matrix $M=(m_{ij})_{n\times n}$ is called a periodic band matrix of bandwidth $b_{n}$ if $m_{ij}=0$ whenever $b_{n}<|i-j|<n-b_{n}$.

$M$ is called a non-periodic band matrix of bandwidth $b_{n}$ if $m_{ij}=0$ whenever $b_{n}<|i-j|$.
\end{dfn}

Notice that in case of a periodic band matrix, the maximum number of non-zero elements in each row is $2b_{n}+1$. On the other hand, in case of a non-periodic band matrix, the number of non-zero elements in a row depends on the index of the row. For example, in the first row there are at most $b_{n}+1$ non-zero elements, and in the $(b_{n}+1)$th row there are at most $2b_{n}+1$ many non-zero elements. In general, the $i$th row of a non-periodic band matrix has at most $b_{n}+i\1_{\{i\leq b_{n}+1\}}+(b_{n}+1)\1_{\{b_{n}+1<i<n-b_{n}\}}+(n+1-i)\1_{\{i\geq n-b_{n}\}}$ many non-zero elements. In any case, the maximum number of non-zero elements is $O(b_{n})$. In this context, let us define two types of index sets. 
 
Let $M=(m_{ij})_{n\times n}$ be a RBM (periodic or non-periodic), then we define 
\begin{align}\label{Chapter Band_ESD:Def: Definition of Ij}
\begin{split}
I_{j}&=\{1\leq k\leq n:m_{jk}\;\text{are not identically zero}\},\\
I_{k}'&=\{1\leq j\leq n:m_{jk}\;\text{are not identically zero}\}.
\end{split}
\end{align}

Notice that in case of periodic band matrices, $|I_{j}|=2b_{n}+1$. Now we proceed to our main results. 

Let $X=(x_{ij})_{n\times n}$ be an $n\times n$ periodic band matrix of bandwidth $b_{n}$, where $b_{n}\to\infty$ as $n\to\infty$. Let $R$ be a sequence of $n\times n$ deterministic periodic band matrices of bandwidth $b_{n}$. Let us denote the ESD of $M$ by $\mu_{M}$. We define $$c_{n}=2b_{n}+1$$ for convenience in writing.  Assume that 
\begin{align}\label{Chapter Band_ESD: Assumptions: main assumptions}
\begin{split}
&(a)\;\text{$\mu_{\f{1}{c_{n}}RR^{*}}$ converges weakly as a measure to $H$, for some non random probability distribution $H$,}\\
&(b)\;\text{$H$ is compactly supported,}\\
&(c)\; \{x_{jk}: \;k\in I_{j},\;1\leq j\leq n\}\;\text{is an iid set of random variables},\\
& (d)\; \E[x_{11}]=0,\E[|x_{11}|^{2}]=1.
\end{split}
\end{align}

Define
\begin{align}\label{Chapter Band_ESD:Def: (Construction) Definition of band matrix}
Y=\f{1}{\sqrt{c_{n}}}(R+\s X),\;\text{where $\s>0$ is fixed.}
\end{align}
For notational convenience, we assume that the band matrix is periodic. However, the following results can easily be extended to the case when the band matrix is not periodic. We will give the outline of the proof in the section \ref{Chapter Band_ESD:Section:Extension of the results to non-periodic band matrices}.

Let $M$ be an $n\times n$ matrix. For convenience, let us introduce the following notation 
\Bea
\{\l_{i}(M):1\leq i\leq n\}&=&\text{eigenvalues of $M$},\\
m_{j}&:=&(m_{1j},m_{2j},\ldots,m_{nj})^{T}\\
\Eea
 It is easy to see that $MM^{*}=\sum_{j=1}^{n}m_{j}m_{j}^{*}$.
 
\begin{dfn}[Poincar\'e inequality]
Let $X$ be a $\R^{d}$ valued random variable with probability measure $\mu$. The random variable $X$ is said to satisfy the Poincar\'e inequality with constant $\kappa>0$, if for all continuously differentiable functions $f:\R^{d}\to\R$,
\Bea
\text{Var}(f(X))\leq\f{1}{\kappa}\E(|\nabla f (X)|^{2}).
\Eea
\end{dfn} 
It can be shown that if $\mu$ satisfies the Poincar\'e inequality with constant $\kappa$, then $\mu\otimes\mu$ also satisfies the Poincar\'e inequality with the same constant $\kappa$ \cite[Theorem 2.5]{guionnet1801lectures}. It can also be shown that if $\mu$ satisfies Poincar\'e inequality and $f:\R^{d}\to \R$ is a continuously differentiable function then 
\bea\label{Chapter Band_ESD:eqn: Anderson tail bound estimate of Poincare random variables}
\Pb_{\mu}\left(|f-\E_{\mu}(f)|>t\right)\leq 2K\exp\left(-\f{\sqrt{\kappa}}{\sqrt{2}\|\|\nabla f\|_{2}\|_{\infty}}t\right),
\eea
where $K=-\sum_{i\geq 0}2^{i}\log(1-2^{-2i-1})$, and $\nabla f$ denotes the gradient of the function $f$. A proof of the above fact can be found in \cite[Lemma 4.4.3]{anderson2010introduction}.

For example, the Gaussian distribution satisfies the Poincar\'e inequality.

\begin{thm}\label{Chapter Band_ESD:Thm: ESD of singular values of random band matrices (with Poincare)}
Let $Y$ be defined in \eqref{Chapter Band_ESD:Def: (Construction) Definition of band matrix}. In addition to the  assumptions made in \eqref{Chapter Band_ESD: Assumptions: main assumptions}, assume that 
\Bea
&&(i)\;\f{(\log n)^{2}}{c_{n}}\to 0,\\
&&(ii)\;\text{Both $\Re(x_{ij})$ and $\Im(x_{ij})$ satisfy Poincar\'e inequality with constant $m$.}
\Eea
Then there exists a non-random probability measure $\mu$ such that $\E|m_{n}(z)-m(z)|\to 0$ uniformly for all $z\in\{z:\Im(z)>\eta\}$ for any fixed $\eta>0$, where $m_{n}(z)=\f{1}{n}\sum_{i=1}^{n}(\l_{i}(YY^{*})-z)^{-1}$ is the Stieltjes transform of ESD of $YY^{*}$, and $m(z)=\int_{\R}\f{d\mu(x)}{x-z}$. In particular, the expected ESD of $YY^{*}$ converges weakly as a measure. In addition, $m(z)$ satisfies
\bea\label{Chapter Band_ESD:eqn: master INTEGRAL equation which is satisfied by m}
m(z)=\int_{\R}\f{dH(t)}{\f{t}{1+\s^{2}m(z)}-(1+\s^{2}m(z))z}\;\;\;\;\text{for any $z\in \C^{+}$}.
\eea
\end{thm}
In particular, the above result is true for standard Gaussian random variables. The Poincar\'e inequality in the Theorem \ref{Chapter Band_ESD:Thm: ESD of singular values of random band matrices (with Poincare)} simplifies the proof a lot. A similar result can also be obtained without Poincar\'e. However in that case, we prove the Theorem under the assumption that the bandwidth grows sufficiently faster. The Theorem is formulated below.

\begin{thm}\label{Chapter Band_ESD:Thm: ESD of singular values of random band matrices (without Poincare)}
Let $Y$ be defined in \eqref{Chapter Band_ESD:Def: (Construction) Definition of band matrix}. In addition to the assumptions made in \eqref{Chapter Band_ESD: Assumptions: main assumptions}, assume that 
\Bea
&&(i)\;\f{n}{c_{n}^{2}}\to 0,\\
&&(ii)\;\E[|x_{11}|^{4p}]<\infty,\; \text{for some}\; p\in \N.
\Eea
Then there exists a non-random probability measure $\mu$ such that $\E|m_{n}(z)-m(z)|^{2p}\to 0$ uniformly for all $z\in\{z:\Im(z)>\eta\}$ for any fixed $\eta>0$, and the Stieltjes transform of $\mu$ satisfies \eqref{Chapter Band_ESD:eqn: master INTEGRAL equation which is satisfied by m}.
\end{thm}
Moreover, if $c_{n}=n^{\alpha}$, where $\alpha>0$, then the $m_{n}(z)$ in Theorem \ref{Chapter Band_ESD:Thm: ESD of singular values of random band matrices (with Poincare)} converges almost surely to $m(z)$. And the same is true for Theorem \ref{Chapter Band_ESD:Thm: ESD of singular values of random band matrices (without Poincare)}, when when $c_{n}=n^{\beta}$ where $\beta=\f{1}{2}+\f{1}{2p}$. We will prove it at the end of the sections \ref{Cahpter Band_ESD:section: main proof of the theorem} and \ref{Chapter Band_ESD:section: Proof of the theorem (with poincare)} respectively.

Notice that if we take $R=0$ and $\sigma=1$, then $H$ is supported only at the real number $0$. In that case \eqref{Chapter Band_ESD:eqn: master INTEGRAL equation which is satisfied by m}, becomes 
\Bea
m(z)(1+m(z))z+1=0,
\Eea
which is the same quadratic equation satisfied by the Stieltjes transform of Marchenko-Pastur law.

Proof of the Theorem \ref{Chapter Band_ESD:Thm: ESD of singular values of random band matrices (without Poincare)} contains the main idea of the proof of both of the Theorems. Main structure of the proof is similar to the method described in \cite{dozier2007empirical}. However in case of band matrices, we need to proof a generalised version of the Lemma 3.1 in \cite{dozier2007empirical}, which is proven in the Propositions \ref{Chapter Band_ESD:Prop: Bound on the difference between trace and quadratic form (without Poincare)} and \ref{Chapter Band_ESD:Prop: Bound on the difference between trace and quadratic form (with Poincare)}. In addition, Lemma \ref{Chapter Band_ESD:Lem: Norm of Y is bounded} gives a large deviation estimate of the norm of a RBM. 

Also, the assumption that $H$ is compactly supported can be weakened by truncating the singular values of $R$ at a threshold of $\log(c_{n})$ and have the same result as the Theorems \ref{Chapter Band_ESD:Thm: ESD of singular values of random band matrices (with Poincare)} and \ref{Chapter Band_ESD:Thm: ESD of singular values of random band matrices (without Poincare)}. But, in that case we need the band width $c_{n}$ to grow a little faster, $\log(c_{n})$ times faster than the existing rate of divergence. We will prove it in the section \ref{Chapter Band_ESD:Section: Truncation of R}.

\section{Proof of Theorem \ref{Chapter Band_ESD:Thm: ESD of singular values of random band matrices (without Poincare)}}\label{Cahpter Band_ESD:section: main proof of the theorem}

Let us define the empirical Stieltjes transform of $YY^{*}$ as $m_{n}=\f{1}{n}\sum_{i=1}^{n}(\l_{i}(YY^{*})-z)^{-1}$. It is clear from the context that $m_{n}$ depends on $z$. So we omit it hereafter to avoid unnecessary cluttering. We introduce the following notations which will be used in the proof of the Theorems.
\eq{\label{Chapter Band_ESD:Eqn:Definitions of A, B, C}
\begin{split}
A&=\f{RR^{*}}{{c_{n}}(1+\s^{2}m_{n})}-\s^{2}zm_{n}I\\
B&=A-zI\\
C&=YY^{*}-zI\\
C_{j}&=C-y_{j}y_{j}^{*}\\
m_{n}^{(j)}&=\f{1}{n}\sum_{i=1}^{n}\left[\l_{i}(YY^{*}-y_{j}y_{j}^{*})-zI\right]^{-1}=\f{1}{n}\sum_{i=1}^{n}(\l_{i}(C_{j}))^{-1}\\
A_{j}&=\f{RR^{*}}{c_{n}(1+\s^{2}m_{n}^{(j)})}-\s^{2}zm_{n}^{(j)}I\\
B_{j}&=A_{j}-zI.
\end{split}
}
Since $YY^{*}=\sum_{j=1}^{n}y_{j}y_{j}^{*}$, we observe that $m_{n}^{(j)}, A_{j},B_{j},C_{j}$ are independent of $y_{j}$. This fact is crucial in our proofs, in particular, in the proof of Proposition \ref{Chapter Band_ESD:Prop: Bound on the difference between trace and quadratic form (without Poincare)}.

\begin{rem}
We notice that the eigenvalues of $A-zI$ are given by $\l_{i}/(1+\s^{2}m_{n})-(1+\s^{2}m_{n})z$, where $\l_{i}$s are eigenvalue of $\f{1}{c_{n}}RR^{*}$. Therefore $\int_{\R}\left[t/(1+\s^{2}m)-(1+\s^{2}m)z\right]^{-1}\;dH(t)$ can be thought of as $\f{1}{n}\tr(A-zI)^{-1}$ for large $n$. So heuristically, proving the Theorem is equivalent to showing that $\f{1}{n}\tr(A-zI)^{-1}-m_{n}\to 0$ as $n\to\infty$. 
\end{rem}

Using the definition \eqref{Chapter Band_ESD:Eqn:Definitions of A, B, C} and Lemma \ref{Chapter Band_ESD:Lem:Sherman-Morrison formula}, we obtain
\Bea
I+zC^{-1}&=&YY^{*}C^{-1}\\
&=&\sum_{j=1}^{n}y_{j}y_{j}^{*}C^{-1}\\
&=&\sum_{j=1}^{n}y_{j}\f{y_{j}^{*}C_{j}^{-1}}{1+y_{j}^{*}C_{j}^{-1}y_{j}}.
\Eea
Taking trace and dividing by $n$ on the both sides, we obtain
\bea\label{Chapter Band_ESD:eqn: stieltjes transform in terms of alphaj}
zm_{n}&=&\f{1}{n}\sum_{j=1}^{n}\f{y_{j}^{*}C_{j}^{-1}y_{j}}{1+y_{j}C_{j}^{-1}y_{j}^{*}}-1\nonumber\\
&=&-\f{1}{n}\sum_{j=1}^{n}\f{1}{1+y_{j}^{*}C_{j}^{-1}y_{j}}.
\eea
Using the resolvent identity,
\Bea
B^{-1}-C^{-1}&=&B^{-1}(YY^{*}-A)C^{-1}\\
&=&\f{1}{c_{n}}B^{-1}\left[RR^{*}+\s RX^{*}+\s XR^{*}+\s^{2}XX^{*}-\f{1}{1+\s^{2}m_{n}}RR^{*}+c_{n}\s^{2}zm_{n}\right]C^{-1}\\
&=&\f{1}{c_{n}}\sum_{j=1}^{n}B^{-1}\left[\f{\s^{2}m_{n}}{1+\s^{2}m_{n}}r_{j}r_{j}^{*}+\s r_{j}x_{j}^{*}+\s x_{j}r_{j}^{*}+\s^{2}x_{j}x_{j}^{*}-\f{c_{n}}{n}\f{1}{1+y_{j}^{*}C_{j}^{-1}y_{j}}\s^{2}\right]C^{-1}.
\Eea
Taking the trace, dividing by $n$, and using \eqref{Chapter Band_ESD:eqn: stieltjes transform in terms of alphaj}, we have
\bea\label{Chapter Band_ESD:eqn: difference between two stieltjes transform is written as sum of five terms}
\f{1}{n}\text{tr}B^{-1}-m_{n}&=&\f{1}{n}\sum_{j=1}^{n}\left[\f{\s^{2}m_{n}}{1+\s^{2}m_{n}}\f{1}{c_{n}}r_{j}^{*}C^{-1}B^{-1}r_{j}+\f{1}{c_{n}}\s x_{j}^{*}C^{-1}B^{-1}r_{j}+\f{1}{c_{n}}\s r_{j}^{*}C^{-1}B^{-1}x_{j}\right.\nonumber\\
&&\left.+\f{1}{c_{n}}\s^{2} x_{j}^{*}C^{-1}B^{-1}x_{j}-\f{1}{1+y_{j}^{*}C_{j}^{-1}y_{j}}\f{1}{n}\s^{2}\tr C^{-1}B^{-1}\right]\nonumber\\
&\equiv&\f{1}{n}\sum_{j=1}^{n}\left[T_{1,j}+T_{2,j}+T_{3,j}+T_{4,j}+T_{5,j}\right].
\eea

For convenience of writing $T_{i,j}$s, let us introduce some notations
\begin{align}\label{Chapter Band_ESD:eqn: Definitions of rho, omega etc.}
\begin{split}
\rho_{j}&=\f{1}{c_{n}}r_{j}^{*}C_{j}^{-1}r_{j},\;\;\;\omega_{j}=\f{1}{c_{n}}\s^{2}x_{j}^{*}C_{j}^{-1}x_{j},\\
\beta_{j}&=\f{1}{c_{n}}\s r_{j}^{*}C_{j}^{-1}x_{j},\;\;\;\g_{j}=\f{1}{c_{n}}\s x_{j}^{*}C_{j}^{-1}r_{j},\\
\hat{\rho}_{j}&=\f{1}{c_{n}}r_{j}^{*}C_{j}^{-1}B^{-1}r_{j},\;\;\;\hat{\omega}_{j}=\f{1}{c_{n}}\s^{2}x_{j}^{*}C_{j}^{-1}B^{-1}x_{j},\\
\hat{\beta_{j}}&=\f{1}{c_{n}}\s r_{j}^{*}C_{j}^{-1}B^{-1}x_{j},\;\;\;\hat{\g}_{j}=\f{1}{c_{n}}\s x_{j}^{*}C_{j}^{-1}B^{-1}r_{j},\\
\alpha_{j}&=1+\f{1}{c_{n}}(r_{j}+\s x_{j})^{*}C_{j}^{-1}(r_{j}+\s x_{j})=1+\rho_{j}+\beta_{j}+\g_{j}+\omega_{j}.
\end{split}
\end{align}

Using Lemma \ref{Chapter Band_ESD:Lem:Sherman-Morrison formula} for $C=C_{j}+y_{j}y_{j}^{*}=C_{j}+\f{1}{c_{n}}(r_{j}+\s x_{j})(r_{j}+\s x_{j})^{*}$ and the above notations, we can compute
\Bea
T_{1,j}&=&\f{1}{c_{n}}\f{\s^{2}m_{n}}{1+\s^{2}m_{n}}\left[r_{j}^{*}C_{j}^{-1}B^{-1}r_{j}-\f{1}{\alpha_{j}}r_{j}^{*}C_{j}^{-1}y_{j}y_{j}^{*}C_{j}^{-1}B^{-1}r_{j}\right]\\
&=&\f{1}{c_{n}\alpha_{j}}\f{\s^{2}m_{n}}{1+\s^{2}m_{n}}\left[\alpha_{j} r_{j}^{*}C_{j}^{-1}B^{-1}r_{j}-\f{1}{c_{n}}r_{j}^{*}C_{j}^{-1}(r_{j}r_{j}^{*}+\s r_{j}x_{j}^{*}+\s x_{j} r_{j}^{*}+\s^{2}x_{j}x_{j}^{*})C_{j}^{-1}B^{-1}r_{j}\right]\\
&=&\f{1}{\alpha_{j}}\f{\s^{2}m_{n}}{1+\s^{2}m_{n}}\left[\alpha_{j} \hat{\rho}_{j}-(\rho_{j}\hat{\rho}_{j}+\rho_{j}\hat{\gamma}_{j}+\beta_{j}\hat{\rho}_{j}+\beta_{j}\hat{\gamma}_{j})\right]\\
&=&\f{1}{\alpha_{j}}\f{\s^{2}m_{n}}{1+\s^{2}m_{n}}[(1+\gamma_{j}+\omega_{j})\hat{\rho}_{j}-(\rho_{j}+\beta_{j})\hat{\gamma}_{j}].\\
\text{Similarly,}&&\\
T_{2j}&=&\f{1}{\alpha_{j}}[(1+\rho_{j}+\beta_{j})\hat{\gamma}_{j}-(\gamma_{j}+\omega_{j})\hat{\rho}_{j}],\\
T_{3,j}&=&\f{1}{\alpha_{j}}[(1+\gamma_{j}+\omega_{j})\hat{\beta}_{j}-(\rho_{j}+\beta_{j})\hat{\omega}_{j}],\\
T_{4,j}&=&\f{1}{\alpha_{j}}[(1+\rho_{j}+\beta_{j})\hat{\omega}_{j}-(\gamma_{j}+\omega_{j})\hat{\beta}_{j}],\\
\text{and},&&\\
T_{5,j}&=&-\f{1}{\alpha_{j}}\f{1}{n}\s^{2}\tr C^{-1}B^{-1}.
\Eea
Using the equations \eqref{Chapter Band_ESD:eqn: stieltjes transform in terms of alphaj} and \eqref{Chapter Band_ESD:eqn: difference between two stieltjes transform is written as sum of five terms} and the above expressions, we can write
\bea\label{Chapter Band_ESD:eqn: AAAAAA The pivotal equation of estimates}
\f{1}{n}\tr B^{-1}-m_{n}&=&\f{1}{n}\sum_{i=1}^{n}\f{1}{\alpha_{j}}\left[\f{1}{1+\s^{2}m_{n}}(\s^{2}m_{n}-\gamma_{j}-\omega_{j})\hat{\rho}_{j}\right.\nonumber\\
&&+\left.\f{1}{1+\s^{2}m_{n}}(1+\rho_{j}+\beta_{j}+\s^{2}m_{n})\hat{\gamma}_{j}+\hat{\beta}_{j}+\hat{\omega}_{j}-\f{1}{n}\s^{2}\tr C^{-1}B^{-1}\right].
\eea

We would like to show that the above quantity converges to zero as $n\to \infty$. Now, we start listing up some basic observations.

Since $x_{ij}$ are iid and $\E[|x_{ij}|^{2}]=1$, by the strong law of large numbers,
\Bea
\f{1}{nc_{n}}\text{tr}XX^{*}=\f{1}{nc_{n}}\sum_{i,j}|x_{ij}|^{2}\stackrel{a.s.}{\to}1.
\Eea
So, $\mu_{\f{1}{c_{n}}XX^{*}}$ is almost surely tight. Using the condition \eqref{Chapter Band_ESD: Assumptions: main assumptions}$(a)$ and Lemma \ref{Chapter Band_ESD:lem: singular value of sum of matrices} we conclude that $\mu_{YY^{*}}$ is almost surely tight. Therefore,
\begin{align*}
\d:=\inf_{n}\int \f{1}{|\l-z|^{2}}d\mu_{YY^{*}}(\l)>0.
\end{align*}
As a result, for any $z\in \C^{+}$, we have
\begin{align}\label{Chapter Band_ESD:eqn: Imzm_n is positive}
\begin{split}
\Im (zm_{n})&=\int \f{\l\Im(z)}{|\l-z|^{2}}\;d\mu_{MM^{*}}(\l)\geq 0,\\
\Im (m_{n})&=\int\f{\Im(z)}{|\l-z|^{2}}\;d\mu_{MM^{*}}(\l)\geq \Im(z)\d> 0.
\end{split}
\end{align}

Let $z\in \C^{+}$, where $\Im(z)$ stands for the imaginary part of $z$. For any Hermitian matrix $M$, $\|(M-zI)^{-1}\|\leq \f{1}{\Im (z)}$. Therefore
\bea\label{Chapter Band_ESD:eqn: bound on the spectral norm of C_j}
\|C^{-1}\|\leq \f{1}{\Im(z)},\;\;\;\|C_{j}^{-1}\|\leq\f{1}{\Im(z)}.
\eea

We also have a similar bound for $B^{-1}$. If $\l$ is an eigenvalue of $\f{1}{c_{n}}RR^{*}$, then $\l(B):=\f{1}{1+\s^{2}m_{n}}\l-(1+\s^{2}m_{n})z$ is the corresponding eigenvalue of $B$. So
\Bea
|\l(B)|\geq|\Im \l(B)|=\left|\f{\s^{2}\Im (m_{n})}{|1+\s^{2}m_{n}|^{2}}\l+\s^{2}\Im (zm_{n})+\Im(z)\right|\geq \Im(z),
\Eea
where the last inequality follows from \eqref{Chapter Band_ESD:eqn: Imzm_n is positive}.

We can do the similar calculations for $B_{j}$. As a result we have
\bea\label{Chapter Band_ESD:eqn: norm of B^-1 is bounded by the imaginary part}
\|B^{-1}\|\leq\f{1}{\Im(z)},\;\;\;\|B_{j}^{-1}\|\leq \f{1}{\Im(z)}.
\eea
Secondly, we would like to estimate the effect of rank one perturbation on $C$ and $B$. More precisely, we would like to estimate $C^{-1}-C_{j}^{-1}$ and $B^{-1}-B_{j}^{-1}$. Using the Lemma \ref{Chapter Band_ESD:Lem: Difference between traces of rank one perturbed matrix is bounded}, we have 
\begin{align}\label{Chapter Band_ESD:eqn: Estimate of m_n-m_nj}
\begin{split}
&\left|\text{tr}(C^{-1}-C_{j}^{-1})\right|\leq\f{1}{|\Im(z)|},\\
&\left|m_{n}-m_{n}^{(j)}\right|=\f{1}{n}\left|\text{tr}(C^{-1}-C_{j}^{-1})\right|\leq\f{1}{n|\Im(z)|}.
\end{split}
\end{align}
Using the estimates \eqref{Chapter Band_ESD:eqn: Imzm_n is positive} for $z\in \C^{+}$, we have 
\Bea
|1+\s^{2}m_{n}|=\f{|z+\s^{2}zm_{n}|}{|z|}\geq\f{1}{|z|}|\Im(z)+\s^{2}\Im(zm_{n})|\geq\f{\Im(z)}{|z|}.
\Eea
Similarly, we also have $|1+\s^{2}m_{n}^{(j)}|\geq \f{\Im(z)}{|z|}$ for $z\in \C^{+}$. 

Therefore, using the estimates \eqref{Chapter Band_ESD:eqn: norm of B^-1 is bounded by the imaginary part},\eqref{Chapter Band_ESD:eqn: Estimate of m_n-m_nj} and the estimate of $\|RR^{*}\|$ from subsection \ref{Chapter Band_ESD:Subsection: Estimate of rhos} we have
\bea\label{Chapter Band_ESD:eqn: Estimate of B^-1-B_j^-1}
\|B^{-1}-B_{j}^{-1}\|&=&\|B^{-1}(B_{j}-B)B_{j}^{-1}\|\nonumber\\
&\leq&\f{1}{|\Im(z)|^{2}}\|B_{j}-B\|\nonumber\\
&=&|m_{n}-m_{n}^{(j)}|\f{\s^{2}}{|\Im(z)|^{2}}\left\|\f{1}{c_{n}(1+\s^{2}m_{n})(1+\s^{2}m_{n}^{(j)})}RR^{*}+zI\right\|\nonumber\\
&\leq&\f{K\s^{2}}{n}.
\eea
Here and in the following estimates, $K>0$ is a constant that depends only on $p,\Im(z)$, and the moments of $x_{ij}$. 

Now, we start estimating several components of the equation \eqref{Chapter Band_ESD:eqn: AAAAAA The pivotal equation of estimates}.

\subsection{Estimates of $\hat{\rho}_{j}$ and $\rho_{j}$}\label{Chapter Band_ESD:Subsection: Estimate of rhos} According to our assumptions we have $\mu_{\f{1}{c_{n}}RR^{*}}\to H$, where $H$ is compactly supported. Therefore, there exists $K>0$ such that
\bea\label{Chapter Band_ESD:eqn: bound on the vector norm of rj}
\|r_{j}\|^{2}=\|r_{j}r_{j}^{*}\|\leq\|RR^{*}\|\leq K c_{n}.
\eea
Using the estimates \eqref{Chapter Band_ESD:eqn: bound on the spectral norm of C_j} and \eqref{Chapter Band_ESD:eqn: norm of B^-1 is bounded by the imaginary part}, we have
\Bea
|\hat{\rho}_{j}|\leq K c_{n},\;\;\;|\rho_{j}|\leq Kc_{n},
\Eea
where $K>0$ is a constant which depends only on the imaginary part of $z$.

\subsection{Estimates of $\gamma_{j}, \beta_{j},\hat{\gamma}_{j}$ and $\hat{\beta}_{j}$} Using Proposition \ref{Chapter Band_ESD:Prop: Bound on the difference between trace and quadratic form (without Poincare)} and equations \eqref{Chapter Band_ESD:eqn: bound on the spectral norm of C_j},\eqref{Chapter Band_ESD:eqn: norm of B^-1 is bounded by the imaginary part}, \eqref{Chapter Band_ESD:eqn: bound on the vector norm of rj}, we have
\Bea
\E[|\gamma_{j}|^{4p}]&=&\f{1}{c_{n}^{4p}}\E\left|x_{j}^{*}C_{j}^{-1}r_{j}r_{j}^{*}(C_{j}^{-1})^{*}x_{j}\right|^{2p}\\
&\leq&\f{K}{c_{n}^{4p}}\E\left|x_{j}^{*}C_{j}^{-1}r_{j}r_{j}^{*}(C_{j}^{-1})^{*}x_{j}-\f{c_{n}}{n}\text{tr}(C_{j}^{-1}r_{j}r_{j}^{*}(C_{j}^{-1})^{*})\right|^{2p}+\f{K}{c_{n}^{2p}n^{2p}}\E\left|r_{j}^{*}C_{j}^{-1}C_{j}^{-1*}r_{j}\right|^{2p}\\
&\leq&\f{Kn^{p}}{c_{n}^{4p}}\|r_{j}r_{j}^{*}\|^{2p}+\f{K}{c_{n}^{2p}n^{2p}|\Im(z)|^{4p}}\|r_{j}\|^{4p}\leq \f{Kn^{p}}{c_{n}^{2p}}+\f{1}{n^{2p}|\Im(z)|^{4p}}\leq \f{Kn^{p}}{c_{n}^{2p}}.
\Eea
Similarly, we can show that 
\Bea
\E[|\beta_{j}|^{4p}]\leq \f{Kn^{p}}{c_{n}^{2p}}.
\Eea

Notice that there are $c_{n}$ many non-trivial elements in the vector $x_{j}$ and $\E[|x_{11}|^{2}]$=1. Therefore $\E\|x_{j}\|^{2}=c_{n}$. Similarly, $$\E\|x_{j}\|^{2p}\leq K c_{n}^{p}.$$ To estimate $\hat{\gamma}_{j}$, we are going to use Proposition \ref{Chapter Band_ESD:Prop: Bound on the difference between trace and quadratic form (without Poincare)}, and equations \eqref{Chapter Band_ESD:eqn: bound on the spectral norm of C_j},\eqref{Chapter Band_ESD:eqn: norm of B^-1 is bounded by the imaginary part} \eqref{Chapter Band_ESD:eqn: bound on the vector norm of rj}, \eqref{Chapter Band_ESD:eqn: Estimate of B^-1-B_j^-1}.
\Bea
\E\left|\hat{\gamma}_{j}\right|^{4p}&=&\f{1}{c_{n}^{4p}}\E\left|x_{j}^{*}C_{j}^{-1}B^{-1}r_{j}\right|^{4p}\\
&\leq&\f{K}{c_{n}^{4p}}\E\left|x_{j}^{*}C_{j}^{-1}B_{j}^{-1}r_{j}\right|^{4p}+\f{K}{c_{n}^{4p}}\E\left|x_{j}^{*}C_{j}^{-1}(B^{-1}-B_{j}^{-1})r_{j}\right|^{4p}\\
&=&\f{K}{c_{n}^{4p}}\E\left|x_{j}^{*}C_{j}^{-1}B_{j}^{-1}r_{j}r_{j}^{*}B_{j}^{-1*}C_{j}^{-1*}x_{j}\right|^{2p}+\f{Kc_{n}^{2p}c_{n}^{2p}}{(nc_{n})^{4p}}\\
&\leq&\f{K}{c_{n}^{4p}}\E\left|x_{j}^{*}C_{j}^{-1}B_{j}^{-1}r_{j}r_{j}^{*}B_{j}^{-1*}C_{j}^{-1*}x_{j}-\f{c_{n}}{n}\text{tr}(C_{j}^{-1}B_{j}^{-1}r_{j}r_{j}^{*}B_{j}^{-1*}C_{j}^{-1*})\right|^{2p}\\
&&+\f{K}{c_{n}^{2p}n^{2p}}\E\left|\text{tr}(C_{j}^{-1}B_{j}^{-1}r_{j}r_{j}^{*}B_{j}^{-1*}C_{j}^{-1*})\right|^{2p}+\f{K}{n^{4p}}\\
&\leq&\f{Kn^{p}}{c_{n}^{2p}}+\f{K}{n^{2p}}+\f{K}{n^{4p}}\leq \f{Kn^{p}}{c_{n}^{2p}}.
\Eea 

Similarly,

\Bea
\E[|\hat{\beta}_{j}|^{4p}]\leq \f{Kn^{p}}{c_{n}^{2p}}.
\Eea

\subsection{Estimates of $\omega_{j}$ and $\hat{\omega}_{j}$}
Using the Proposition \ref{Chapter Band_ESD:Prop: Bound on the difference between trace and quadratic form (without Poincare)}, Lemma \ref{Chapter Band_ESD:Lem: Difference between traces of rank one perturbed matrix is bounded} and the estimates \eqref{Chapter Band_ESD:eqn: bound on the spectral norm of C_j}, \eqref{Chapter Band_ESD:eqn: norm of B^-1 is bounded by the imaginary part}, \eqref{Chapter Band_ESD:eqn: Estimate of m_n-m_nj}, \eqref{Chapter Band_ESD:eqn: Estimate of B^-1-B_j^-1}, we can write
\Bea
\f{1}{\s^{4p}}\E\left|\hat{\omega}_{j}-\f{\s^{2}}{n}\text{tr}C^{-1}B^{-1}\right|^{2p}&=&\f{1}{\s^{4p}}\E\left|\f{1}{c_{n}}\s^{2}x_{j}^{*}C_{j}^{-1}B^{-1}x_{j}-\f{\s^{2}}{n}\text{tr}C^{-1}B^{-1}\right|^{2p}\\
&\leq&\f{K}{c_{n}^{2p}}\E\left|x_{j}^{*}C_{j}^{-1}(B^{-1}-B_{j}^{-1})x_{j}\right|^{2p}+\f{K}{c_{n}^{2p}}\E\left|x_{j}^{*}C_{j}^{-1}B_{j}^{-1}x_{j}-\f{c_{n}}{n}\text{tr}C_{j}^{-1}B_{j}^{-1}\right|^{2p}\\
&&+\f{K}{n^{2p}}\E\left|\text{tr}(C^{-1}-C_{j}^{-1})B^{-1}\right|^{2p}+\f{K}{n^{2p}}\E\left|\text{tr}C_{j}^{-1}(B^{-1}-B_{j}^{-1})\right|^{2p}\\
&\leq&\f{K}{c_{n}^{2p}n^{2p}}\E\|x_{j}\|^{2p}+\f{Kn^{p}}{c_{n}^{2p}}+\f{K}{n^{2p}}+\f{K}{n^{2p}}\leq\f{Kn^{p}}{c_{n}^{2p}}.
\Eea
Similarly, it can be shown that 
\Bea
\f{1}{\s^{4p}}\E\left|\omega_{j}-\s^{2}m_{n}\right|^{2p}=\f{1}{\s^{4p}}\E\left|\omega_{j}-\f{\s^{2}}{n}\text{tr}C^{-1}\right|^{2p}\leq\f{Kn^{p}}{c_{n}^{2p}}.
\Eea

This completes the estimates of the main components of \eqref{Chapter Band_ESD:eqn: AAAAAA The pivotal equation of estimates}. Finally, we notice that if $z\in \C^{+}$, then $\Im(zy_{j}^{*}(C_{j}-zI)^{-1}y_{j})\geq 0$. As a result, we have $|z\alpha_{j}|\geq |\Im(z)|$. 

Plugging in all the above estimates into \eqref{Chapter Band_ESD:eqn: AAAAAA The pivotal equation of estimates}, we obtain
\Bea
\E\left|\f{1}{n}\tr B^{-1}-m_{n}\right|^{2p}&\leq&\f{K}{n}\sum_{j=1}^{n}\f{Kn^{p}}{c_{n}^{2p}}\leq\f{Kn^{p}}{c_{n}^{2p}}\to 0.
\Eea 
Since $|m_{n}|\leq\f{1}{\Im(z)}$, there exists a subsequence $\{m_{n_{k}}\}_{k}$ such that $\{m_{n_{k}}\}_{k}$ converges. Uniqueness of the solution of \eqref{Chapter Band_ESD:eqn: master INTEGRAL equation which is satisfied by m} can be proved in the exact same way as described in \cite[Section 4]{dozier2007empirical}. Also following the same exact procedure as described in \cite[End of section 3]{dozier2007empirical}, it can be proved that
\Bea
\f{1}{n}\tr B_{n_{k}}^{-1}\to \int\f{dH(t)}{\f{t}{1+\s^{2}m(z)}-(1+\s^{2}m(z))z}\;\;\;\text{a.s.}
\Eea
We skip the details here. This completes the proof of the Theorem \ref{Chapter Band_ESD:Thm: ESD of singular values of random band matrices (without Poincare)}.

From the above estimate, we also see that if $c_{n}=n^{\beta}$, where $\beta>\f{1}{2}+\f{1}{2p}$, then $\sum_{n=1}^{\infty}\f{n^{p}}{c_{n}^{2p}}<\infty$. Therefore by Borel-Cantelli Lemma, we can conclude that $\f{1}{n}\tr B^{-1}-m_{n}\to 0$ almost surely. 

\section{Proof of Theorem \ref{Chapter Band_ESD:Thm: ESD of singular values of random band matrices (with Poincare)}}\label{Chapter Band_ESD:section: Proof of the theorem (with poincare)} Proof of this Theorem is exactly same as the proof of Theorem \ref{Chapter Band_ESD:Thm: ESD of singular values of random band matrices (without Poincare)}. We notice that we obtained the bound $O\left(\f{n^{p}}{c_{n}^{2p}}\right)$ using the Proposition \ref{Chapter Band_ESD:Prop: Bound on the difference between trace and quadratic form (without Poincare)}. So while estimating the bounds of several components of equation \eqref{Chapter Band_ESD:eqn: AAAAAA The pivotal equation of estimates}, instead of using the Proposition \ref{Chapter Band_ESD:Prop: Bound on the difference between trace and quadratic form (without Poincare)}, we will use the Proposition \ref{Chapter Band_ESD:Prop: Bound on the difference between trace and quadratic form (with Poincare)}. And by doing so we can obtain that $\E\left|\f{1}{n}\tr B^{-1}-m_{n}\right|^{2}=O(1/c_{n})$. Which will conclude the Theorem \ref{Chapter Band_ESD:Thm: ESD of singular values of random band matrices (with Poincare)}.

To prove the almost sure convergence, we can truncate all the entries of the matrix $X$ at $6\sqrt{\f{2}{\kappa}}\log n$. Let us denote that truncated matrix as $\tilde{X}$. Since $x_{ij}$s satisfy the Poincar\'e inequality, from \eqref{Chapter Band_ESD:eqn: Anderson tail bound estimate of Poincare random variables} we have
\Bea
\Pb\left(|x_{ij}|>t\right)\leq 2K\exp\left(-\sqrt{\f{\kappa}{2}}t\right).
\Eea 
Therefore,
\Bea
\Pb\left(X\neq \tilde{X}\right)\leq 2Kn^{2}\exp\left(-6\log n\right)\leq \f{K}{n^{4}}.
\Eea
Now using the second part of Proposition \ref{Chapter Band_ESD:Prop: Bound on the difference between trace and quadratic form (with Poincare)} and following the same method as described in section \ref{Cahpter Band_ESD:section: main proof of the theorem}, we have
\Bea
\E\left[\left|\f{1}{n}\tr B^{-1}-m_{n}\right|^{2l}\1_{\{X=\tilde{X}\}}\right]\leq K\f{(\log n)^{2l}}{c_{n}^{l}}.
\Eea 
Since $\left|\f{1}{n}\tr B^{-1}\right|,|m_{n}|\leq |\Im z|^{-1}$, we have
\Bea
\E\left[\left|\f{1}{n}\tr B^{-1}-m_{n}\right|^{2l}\right]\leq K\f{(\log n)^{2l}}{c_{n}^{l}}+\f{K}{|\Im z|^{2l}n^{4}}.
\Eea
If $c_{n}=n^{\alpha}$, $\alpha>0$, then taking $l$ large enough and using the Borel-Cantelli Lemma we may conclude the almost sure convergence.

\section{Truncation of $R$}\label{Chapter Band_ESD:Section: Truncation of R} In several estimates, it was convenient when we had bounded $r_{ij}$. However, we can achieve the same results as described in the Theorems \ref{Chapter Band_ESD:Thm: ESD of singular values of random band matrices (without Poincare)}, and Theorem \ref{Chapter Band_ESD:Thm: ESD of singular values of random band matrices (with Poincare)} by truncating the Singular values of $R$. Below, we have described the truncation method by following the same procedure as described in \cite{dozier2007empirical}.

Let $\f{1}{\sqrt{c_{n}}}R=USV$ be the singular value decomposition of $R$, where $S=diag[s_{1},\ldots, s_{n}]$ are the singular values of $R$ and $U$, $V$ are orthonormal matrices. Let us construct a diagonal matrix $S_{\alpha}$ as $S_{\alpha}=diag[s_{1}\1(s_{1}\leq\alpha),\ldots, s_{n}\1(s_{n}\leq\alpha)]$, and consider the matrices $R_{\alpha}=US_{\alpha}V$, $Y_{\alpha}=\f{1}{\sqrt{c_{n}}}(R_{\alpha}+\s X)$. Then by Lemma \ref{Chapter Band_ESD:Lem: CDF is bounded by the rank perturbation}, we have
\Bea
\|\mu_{YY^{*}}-\mu_{Y_{\alpha}Y_{\alpha}^{*}}\|&\leq&\f{2}{n}\text{rank}\left(\f{R}{\sqrt{c_{n}}}-\f{R_{\alpha}}{\sqrt{c_{n}}}\right)\\
&=&\f{2}{n}\sum_{i=1}^{n}\1(s_{i}>\alpha)\\
&=&2H(\alpha^{2},\infty).
\Eea
If we take $\alpha^{2}\to\infty$ for example $\alpha=\log(c_{n})$ then $\|\mu_{YY^{*}}-\mu_{Y_{\alpha}Y_{\alpha}^{*}}\|\to 0$. So without loss of generality we can assume that $\|r_{j}\|^{2}\leq\|RR^{*}\|\leq c_{n}\log(c_{n})$. In that case, we have
\Bea
\|r_{j}\|^{2}=\|r_{j}r_{j}^{*}\|\leq\|RR^{*}\|\leq c_{n}\log(c_{n}).
\Eea
 So, using the estimates \eqref{Chapter Band_ESD:eqn: bound on the spectral norm of C_j} and \eqref{Chapter Band_ESD:eqn: norm of B^-1 is bounded by the imaginary part} we have
\Bea
|\hat{\rho}_{j}|\leq Kc_{n}\log(c_{n}),\;\;\;|\rho_{j}|\leq Kc_{n}\log(c_{n}),
\Eea
where $K>0$ is a constant which depends only on the imaginary part of $z$. Similarly, all the places in the proof of Theorem \ref{Chapter Band_ESD:Thm: ESD of singular values of random band matrices (without Poincare)} we can replace the estimates $|r_{j}r_{j}^{*}|\leq Kc_{n}$ by the estimates $|r_{j}r_{j}^{*}|\leq Kc_{n}\log (c_{n})$.

\section{Extension of the results to non-periodic band matrices}\label{Chapter Band_ESD:Section:Extension of the results to non-periodic band matrices}
The result can easily be extended to non-periodic band matrices. We observe that for the purpose of our proof, the main difference between a periodic and a non-periodic band matrix is the number of elements in certain rows. In the case of a periodic band matrix, the number of non-trivial elements in any row is $|I_{j}|=2b_{n}+1=c_{n}$, which is fixed for any $1\leq j\leq n$. Therefore, in the definition \eqref{Chapter Band_ESD:eqn: Definitions of rho, omega etc.} we divide by $c_{n}$. For a non periodic band matrix $|I_{j}|=b_{n}+i\1_{\{i\leq b_{n}+1\}}+(b_{n}+1)\1_{\{b_{n}+1<i<n-b_{n}\}}(n+1-i)\1_{\{i\geq n-b_{n}\}}=O(b_{n})$. Once in the definition \eqref{Chapter Band_ESD:eqn: Definitions of rho, omega etc.} and in the Proposition \ref{Chapter Band_ESD:Prop: Bound on the difference between trace and quadratic form (without Poincare)}, Proposition \ref{Chapter Band_ESD:Prop: Bound on the difference between trace and quadratic form (with Poincare)} if we replace $c_{n}$ by $|I_{j}|$, everything works out as before.  

\section{Two concentration results}\label{Chapter Band_ESD:Section: Main concentration results} In this section we list two main concentration results which are used in the proofs of the Theorems \ref{Chapter Band_ESD:Thm: ESD of singular values of random band matrices (with Poincare)}, \ref{Chapter Band_ESD:Thm: ESD of singular values of random band matrices (without Poincare)}.

\begin{prop}\label{Chapter Band_ESD:Prop: Bound on the difference between trace and quadratic form (without Poincare)}
Let $M$ be one of $C_{j}^{-1},C_{j}^{-1}B_{j}^{-1}$, and $N$ be one of $C_{j}^{-1}r_{j}r_{j}^{*}C_{j}^{-1*}$ or $C_{j}^{-1}B_{j}^{-1}r_{j}r_{j}^{*}B^{-1*}C_{j}^{-1*}$. Let $x_{j}$ be the $j$th column of $X$ as defined in Theorem \ref{Chapter Band_ESD:Thm: ESD of singular values of random band matrices (without Poincare)}. Let us also assume that $\E|x_{11}|^{4l}<\infty$. Then for any $l\in \N$, 
\Bea
&&\E\left|x_{j}^{*}Mx_{j}-\f{c_{n}}{n}\text{tr}M\right|^{2l}\leq Kn^{l}\\
&&\E\left|x_{j}^{*}Nx_{j}-\f{c_{n}}{n}\text{tr}N\right|^{2l}\leq Kn^{l}\|r_{j}r_{j}^{*}\|^{2l},
\Eea
where $K>0$ is a constant that depends on $l$, $\Im(z)$, and the moments of $x_{j}$, but not on $n$.
\end{prop}

\begin{proof}
From the estimates \eqref{Chapter Band_ESD:eqn: bound on the spectral norm of C_j} and \eqref{Chapter Band_ESD:eqn: norm of B^-1 is bounded by the imaginary part} we know that $\|C_{j}^{-1}\|\leq 1/|\Im(z)|$ and $\|B_{j}^{-1}\|\leq 1/|\Im(z)|$. So for convenience of writing the proof, let us assume that $\|M\|\leq 1$ and $\|N\|\leq \|r_{j}r_{j}^{*}\|$. Also without loss of generality, we can assume that $j=1$, and recall the definition of $I_{j}$ from \eqref{Chapter Band_ESD:Def: Definition of Ij}. We can write $M=P+iQ$, where $P$ and $Q$ are the real and imaginary parts of $M$ respectively. Then we can write
\Bea
\E\left|x_{j}^{*}Mx_{j}-\f{c_{n}}{n}\text{tr}M\right|^{2l}\leq 2^{2l-1}\E\left|x_{1}^{*}Px_{1}-\f{c_{n}}{n}\text{tr}P\right|^{2l}+2^{2l-1}\E\left|x_{1}^{*}Qx_{1}-\f{c_{n}}{n}\text{tr}Q\right|^{2l}.
\Eea
We can write the first part as
\Bea
\left|x_{1}^{*}Px_{1}-\f{c_{n}}{n}\text{tr}P\right|^{2l}&=&\left|x_{1}^{*}Px_{1}-\sum_{k\in I_{1}}P_{kk}+\sum_{k\in I_{1}}P_{kk}-\f{c_{n}}{n}\text{tr}P\right|^{2l}\\
&\leq&3^{2l-1}\E\left[\sum_{k\in I_{1}}(|x_{1k}|^{2}-1)P_{kk}\right]^{2l}+3^{2l-1}\E\left[\sum_{\substack{i\neq j\\ i,j\in I_{1}}}P_{ij}\overline{x_{1i}}x_{1j}\right]^{2l}\\
&&+3^{2l-1}\left|\sum_{k\in I_{1}}P_{kk}-\f{c_{n}}{n}\text{tr}P\right|^{2l}\\
&=:&3^{2l-1}(S_{1}+S_{2}+S_{3}).
\Eea
Following the same procedure as in \cite{silverstein1995empirical}, we can estimate the first part. Note that $\|P^{m}\|\leq\|P\|^{m}\leq\|M\|^{m}\leq 1$ for any $m\in \N$. In the expansion of $\left[\sum_{k\in I_{1}}(|x_{1k}|^{2}-1)P_{kk}\right]^{2l}$, the maximum contribution (in terms of $c_{n}$) will come from the terms like
\Bea
\sum_{k_{1},\ldots, k_{l}\in I_{1}}(|x_{1k_{1}}|^{2}-1)^{2}\cdots(|x_{1k_{l}}|^{2}-1)^{2}(P_{i_{1}i_{1}}\cdots P_{i_{l}i_{l}})^{2},
\Eea
when all $i_{1},\ldots, i_{l}$ are distinct. Note that $(P_{i_{1}i_{1}}\cdots P_{i_{l}i_{l}})^{2}\leq 1$. Consequently, expectation of the above term is bounded by $K c_{n}^{l}$, where $K$ depends only on the fourth moment of $x_{ij}$. Therefore
\Bea
S_{1}=\E\left[\sum_{k\in I_{1}}(|x_{1k}|^{2}-1)P_{kk}\right]^{2l}\leq K c_{n}^{l},
\Eea
where $K$ depends only on $l$ and the moments of $x_{ij}$.

Since $C_{1}^{-1},C_{1}^{-1}B_{1}^{-1},C_{1}^{-1}r_{1}r_{1}^{*}C_{1}^{-1*}$ or $C_{1}^{-1}B_{1}^{-1}r_{1}r_{1}^{*}B^{-1*}C_{1}^{-1*}$ are independent of $x_{1}$, for the second sum we have
\Bea
\sum_{\substack{i_{1}\neq j_{1},\ldots,i_{2l}\neq j_{2l}\\ i_{1},j_{1},\ldots,i_{2l},j_{2l}\in I_{1}}} \E[P_{i_{1}j_{1}}\cdots P_{i_{2l}j_{2l}}]\E[\overline{x_{1i_{1}}}x_{1j_{1}}\cdots \overline{x_{1i_{2l}}}x_{1j_{2l}}].
\Eea
The expectation will be zero if a term appears only once and the maximum contribution (in terms of $c_{n}$) will come from the case when each of  $x_{1j}$ and $\overline{x_{1j}}$ appears only twice. In that case, the contribution is 
\Bea
\sum_{\substack{i_{1}\neq j_{1}\\ i_{1},j_{1}\in I_{1}}}P_{i_{1}j_{1}}^{2}\cdots\sum_{\substack{i_{l}\neq j_{l}\\ i_{l},j_{l}\in I_{1}}}P_{i_{l}j_{l}}^{2}\leq  c_{n}^{l},
\Eea
where the last inequality follows from the fact that $\sum_{i,j\in I_{1}}P_{ij}^{2}=\text{tr}(LPL^{T}LP^{T}L^{T})\leq c_{n}$, where $L_{c_{n}\times n}$ is the projection matrix onto the co-ordinates indexed by $I_{1}$. As a result, we have
\Bea
S_{2}=\E\left[\sum_{\substack{i\neq j\\ i,j\in I_{1}}}P_{ij}\overline{x_{1i}}x_{1j}\right]^{2l}\leq Kc_{n}^{l},
\Eea
where $K$ depends only on $l$ and the moments of $x_{ij}$.

To estimate the $S_{3}$, we can write it as
\Bea
S_{3}=\left|\sum_{k\in I_{1}}P_{kk}-\f{c_{n}}{n}\text{tr}P\right|^{2l}&=&2^{2l-1}\left|\sum_{k\in I_{1}}P_{kk}-\E\sum_{k\in I_{1}}P_{kk}\right|^{2l}+2^{2l-1}\left|\E\sum_{k\in I_{1}}P_{kk}-\f{c_{n}}{n}\text{tr}P\right|^{2l}.
\Eea
Since $|P_{kk}-\E[P_{kk}]|\leq|(C_{1}^{-1})_{kk}-\E[(C_{1}^{-1})_{kk}]|$, from Lemma \ref{Chapter Band_ESD:Lem: Exponential tail bound on the partial trace of the resolvent} we have an exponential tail bound on

 $\left|\sum_{k\in I_{1}}P_{kk}-\E\sum_{k\in I_{1}}P_{kk}\right|$. As a result,
\bea\label{Chapter Band_ESD:Eqn: The equation responsible of n^l order in without poincare proposition}
\E\left|\sum_{k\in I_{1}}P_{kk}-\E\sum_{k\in I_{1}}P_{kk}\right|^{2l}\leq K n^{l},
\eea
where $K$ depends only on $l$.

Since $x_{ij}$ are iid, for any choice of $M$, we have $\E[m_{11}]=\E[m_{ii}]$. Which implies that $\E[\sum_{k\in I_{1}}P_{kk}]=\f{c_{n}}{n}\E[\text{tr}P]$. Therefore from Lemma \ref{Chapter Band_ESD:Lem: Exponential tail bound on the partial trace of the resolvent}, we have
\Bea
\left|\E\sum_{k\in I_{1}}P_{kk}-\f{c_{n}}{n}\text{tr}P\right|^{2l}&=&\f{c_{n}^{2l}}{n^{2l}}\left|\E[\text{tr}P]-\text{tr}P\right|^{2l}\\
&\leq&K\f{c_{n}^{2l}}{n^{l}}\leq K c_{n}^{l},
\Eea
where $K$ depends only on $l$. Hence we have
\Bea
S_{3}\leq K(n^{l}+c_{n}^{l}).
\Eea
Combining all the above estimates, we get 
\Bea
\E\left|x_{1}^{*}Px_{1}-\f{c_{n}}{n}\text{tr}P\right|^{2l}\leq Kn^{l}.
\Eea 
Repeating the above computation, we can do the same estimate $\E\left|x_{1}^{*}Qx_{1}-\f{c_{n}}{n}\text{tr}Q\right|^{2l}\leq K n^{l}$. This completes the proof.\\\\

\end{proof}

\begin{lem}[Norm of a random band matrix]\label{Chapter Band_ESD:Lem: Norm of Y is bounded} 
Let $X$ and $Y$ be defined in \eqref{Chapter Band_ESD:Def: (Construction) Definition of band matrix}, $x_{ij}$ satisfy the Poincar\'e inequality with constant $m$, and $c_{n}>(\log n)^{2}$. Then $\E\|XX^{*}\|\leq K c_{n}^{2}$ for some universal constant $K$ which may depend on the Poincar\'e constant $m$. In particular, if the limiting ESD of $\f{1}{c_{n}}RR^{*}$ i.e., $H$ is compactly supported then $\E\|YY^{*}\|\leq Kc_{n}$.
\end{lem}

\begin{proof}
We will follow the method described in \cite{tropp2015introduction, mackey2014matrix, tropp2012user} and the references therein. The analysis becomes somewhat easier if we assume that all non-zero entries of $X$ are standard Gaussian random variables. However, it contains the main idea of the analysis.\\\\

\textbf{Case I} ($x_{jk}$ are standard Gaussian random variables):  Using the Markov's inequality, we have
\Bea
\Pb\left(\f{1}{c_{n}}\|XX^{*}\|>t\right)\leq e^{- t}\E\left[\exp\left(\f{1}{c_{n}}\|XX^{*}\|\right)\right]\leq e^{- t}\E\left[\tr\exp\left(\f{1}{c_{n}}XX^{*}\right)\right],
\Eea
 To estimate the right hand side, we will use the Lieb's Theorem. Let $H$ be any $n\times n$ fixed Hermitian matrix. From Lieb's Theorem (\cite{lieb1973convex}, Theorem 6), we know that the function $f(A)=\tr\exp(H+\log A)$ is a concave function on the convex cone of $n\times n$ positive definite Hermitian matrices.  

Let us write $\f{1}{c_{n}}XX^{*}=\sum_{k=1}^{n}x_{k}x_{k}^{*}$, where $x_{k}$ is the $k$th column vector of $X/\sqrt{c_{n}}$. Then using Lieb's Theorem and Jensen's inequality, we have
\Bea
\E\left[\left.\tr\exp\left(\f{1}{c_{n}}XX^{*}\right)\right|x_{1},\ldots,x_{n-1}\right]&=&\E\left[\left.\tr\exp\left(\f{1}{c_{n}}\sum_{k=1}^{n-1}x_{k}x_{k}^{*}+\log \exp\left(\f{1}{c_{n}}x_{n}x_{n}^{*}\right)\right)\right|x_{1},\ldots,x_{n-1}\right]\\
&\leq& \tr\exp\left[\f{1}{c_{n}}\sum_{k=1}^{n-1}x_{k}x_{k}^{*}+\log \E \exp\left(\f{1}{c_{n}}x_{n}x_{n}^{*}\right)\right].
\Eea 
Proceeding in this way, we obtain
\Bea
\E\left[\tr\exp\left(\f{1}{c_{n}}XX^{*}\right)\right]\leq \tr\exp\left[\sum_{k=1}^{n}\log\E\exp\left(\f{1}{c_{n}}x_{k}x_{k}^{*}\right)\right].
\Eea
Therefore
\bea\label{Chapter Band_ESD:eqn: (first) derived master formula from Lieb's inequality}
\Pb\left(\f{1}{c_{n}}\|XX^{*}\|>t\right)\leq e^{-t}\tr\exp\left[\sum_{k=1}^{n}\log\E\exp\left(\f{1}{c_{n}}x_{k}x_{k}^{*}\right)\right].
\eea
It is easy to see that
\Bea
\exp\left(\f{1}{c_{n}}x_{k}x_{k}^{*}\right)&=&I+\left(\sum_{l=1}^{\infty}\f{1}{l!c_{n}^{l}}\|x_{k}\|^{2(l-1)}\right)x_{k}x_{k}^{*}\\
&=&I+\f{e^{\|x_{k}\|^{2}/c_{n}}-1}{\|x_{k}\|^{2}}x_{k}x_{k}^{*}\\
&\preceq&I+\f{1}{c_{n}} e^{\|x_{k}\|^{2}/c_{n}}x_{k}x_{k}^{*},
\Eea
where $A\preceq B$ denotes that $(B-A)$ is positive semi-definite. Since $\{x_{jk}\}_{1\leq k\leq n,\;j\in I_{k}'}$ are independent standard Gaussian random variables, we have 
\Bea
\E\left[e^{\|x_{k}\|^{2}/c_{n}}x_{jk}\bar{x}_{lk}\right]=0,\;\;\;\text{if $j\neq l$}\\
\E\left[e^{\|x_{k}\|^{2}/c_{n}}|x_{jk}|^{2}\right]=\left(1-\f{1}{c_{n}}\right)^{-(c_{n}+1)}.
\Eea

As a result,
\Bea
\tr\exp\left[\sum_{k=1}^{n}\log\E\exp\left(\f{1}{c_{n}}x_{k}x_{k}^{*}\right)\right]\leq n\left(1+\f{e}{c_{n}}\right)^{c_{n}}.
\Eea
Substituting this estimate in \eqref{Chapter Band_ESD:eqn: (first) derived master formula from Lieb's inequality}, we have
\bea\label{Chapter Band_ESD:Eqn: Tail estimate of norm of band matrix}
\Pb\left(\f{1}{c_{n}}\|XX^{*}\|>t+\log n\right)\leq e^{e}ne^{-(t+\log n)}=e^{e}e^{-t}.
\eea
As a result,
\Bea
\f{1}{c_{n}}\E[\|XX^{*}\|]&=&\int_{0}^{\infty}\Pb\left(\f{1}{c_{n}}\|XX^{*}\|>u\right)\;du\\
&\leq&\int_{0}^{\log n}\;du+\int_{0}^{\infty}\Pb\left(\f{1}{c_{n}}\|XX^{*}\|>t+\log n\right)\;dt\\
&\leq&\log n+e^{e}\leq K c_{n}.
\Eea
This completes the proof.\\\\

\textbf{Case II} ($x_{jk}$s satisfy the Poincar\'e inequality): First of all, let us write the random matrix $X$ as $X=X_{1}+i X_{2}$, where $X_{1}$ and $X_{2}$ are the real and imaginary parts of $X$ respectively. Since $\|X\|\leq\|X_{1}\|+\|X_{2}\|$, it is enough to estimate $\|X_{1}\|$ and $\|X_{2}\|$ separately. In other words, without loss of generality, we can assume that $x_{ij}$ are real valued random variables. 

Let us construct the matrix 
\Bea
\tilde{X}=\left[\begin{array}{cc}
O & X\\
X & O
\end{array}
\right].
\Eea
It is easy to see that $\|\tilde{X}\|=\|X\|$. Therefore it is enough to bound $\E\|\tilde{X}\|^{2}$. 

We can write $\tilde{X}$ as
\Bea
\tilde{X}=\sum_{i=1}^{n}\sum_{j\in I_{i}}x_{ij}(E_{i,n+j}+E_{n+j,i}),
\Eea
where $E_{ij}$ is a $2n\times 2n$ matrix with all $0$ entries except $1$ at the $(i,j)$th position. Proceeding in the same way as case I, we may write
\bea\label{Chapter Band_ESD:eqn: (second) derived master formula from Lieb's inequality}
\Pb\left(\f{1}{\sqrt{c_{n}}}\|\tilde{X}\|>t\right)\leq e^{-t}\tr\exp\left[\sum_{i=1}^{n}\sum_{j\in I_{i}}\log \E\exp\left(\f{1}{\sqrt{c_{n}}}x_{ij}(E_{i,n+j}+E_{n+j,i})\right)\right].
\eea

Let us consider the $2\times 2$ matrix $H=\left[\begin{array}{cc}0 & \gamma\\ \gamma & 0\end{array}\right]$, where $\gamma$ is a real valued random variable. By the spectral calculus, we have

%
%

\Bea
\log\E[\exp(H)]=\f{1}{2}\left[\begin{array}{cc}
1 & 1\\
1 & -1
\end{array}\right]
\left[\begin{array}{cc}
\log\E e^{\gamma} & 0\\
0 & \log\E e^{-\gamma}
\end{array}\right]
\left[\begin{array}{cc}
1 & 1\\
1 & -1
\end{array}\right]
=\f{1}{2}\left[\begin{array}{cc}
\log[\E e^{\gamma}\E e^{-\gamma}] & \log[\E e^{\gamma}/\E e^{-\gamma}]\\
\log[\E e^{\gamma}/\E e^{-\gamma}] & \log[\E e^{\gamma}\E e^{-\gamma}]
\end{array}\right].
\Eea
Since $x_{ij}$s are iid, let us assume that all $x_{ij}$ have the same probability distribution as a real valued random variable $\gamma$. Then proceeding as above, we can see that
\Bea
\log \E\exp\left(\f{1}{\sqrt{c_{n}}}x_{ij}(E_{i,n+j}+E_{n+j,i})\right)&=&\f{1}{2}\log[\E e^{\gamma/\sqrt{c_{n}}}\E e^{-\gamma/\sqrt{c_{n}}}](E_{ii}+E_{n+j,n+j})\\
&&+\f{1}{2}\log[\E e^{\gamma/\sqrt{c_{n}}}/\E e^{-\gamma/\sqrt{c_{n}}}](E_{i,n+j}+E_{n+j,i}).
\Eea
Therefore,
\Bea
\sum_{i=1}^{n}\sum_{j\in I_{i}}\log \E\exp\left(\f{1}{\sqrt{c_{n}}}x_{ij}(E_{i,n+j}+E_{n+j,i})\right)&=&\f{c_{n}}{2}\log[\E e^{\gamma/\sqrt{c_{n}}}\E e^{-\gamma/\sqrt{c_{n}}}]\;I\\
&&+\f{1}{2}\log[\E e^{\gamma/\sqrt{c_{n}}}/\E e^{-\gamma/\sqrt{c_{n}}}]\sum_{i=1}^{n}\sum_{j\in I_{i}}(E_{i,j+n}+E_{j+n,i}).
\Eea

From Golden–Thompson inequality, if $A$ and $B$ are two $d\times d$ real symmetric matrices then $\tr e^{A+B}\leq \tr(e^{A}e^{B})$.

In our case, let us take
\Bea
A&=&\f{c_{n}}{2}\log[\E e^{\gamma/\sqrt{c_{n}}}\E e^{-\gamma/\sqrt{c_{n}}}]\;I\\
B&=&\f{1}{2}\log[\E e^{\gamma/\sqrt{c_{n}}}/\E e^{-\gamma/\sqrt{c_{n}}}]\sum_{i=1}^{n}\sum_{j\in I_{i}}(E_{i,j+n}+E_{j+n,i}).
\Eea
Then 
\Bea
e^{A}=[\E e^{\gamma/\sqrt{c_{n}}}\E e^{-\gamma/\sqrt{c_{n}}}]^{c_{n}/2}\;I.
\Eea
\Bea
&&\tr\exp\left[\sum_{i=1}^{n}\sum_{j\in I_{i}}\log \E\exp\left(\f{1}{\sqrt{c_{n}}}x_{ij}(E_{i,n+j}+E_{n+j,i})\right)\right]\\
&\leq &\tr\left[\left\{[\E e^{\gamma/\sqrt{c_{n}}}\E e^{-\gamma/\sqrt{c_{n}}}]^{nc_{n}/2}\right\}e^{B}\right]\\
&\leq& \left\{[\E e^{\gamma/\sqrt{c_{n}}}\E e^{-\gamma/\sqrt{c_{n}}}]^{nc_{n}/2}\right\}\;ne^{\|B\|}.
\Eea

It is easy to see that $\left\|\sum_{i=1}^{n}\sum_{j\in I_{i}}(E_{i,j+n}+E_{j+n,i})\right\|\leq c_{n}$. Combining all the estimates and plugging them in \eqref{Chapter Band_ESD:eqn: (second) derived master formula from Lieb's inequality}, we obtain
\Bea
\Pb\left(\f{1}{\sqrt{c_{n}}}\|\tilde{X}\|>t\right)&\leq& ne^{-t}[\E e^{\gamma/\sqrt{c_{n}}}\E e^{-\gamma/\sqrt{c_{n}}}]^{c_{n}/2}[\E e^{\gamma/\sqrt{c_{n}}}/\E e^{-\gamma/\sqrt{c_{n}}}]^{c_{n}/2}\\
&=&ne^{-t}\left\{\E e^{\gamma/\sqrt{c_{n}}}\right\}^{c_{n}}.
\Eea

From the concentration estimate \eqref{Chapter Band_ESD:eqn: Anderson tail bound estimate of Poincare random variables}, we have that $\Pb(|\gamma|>t)\leq \exp\{-t\sqrt{\kappa}/\sqrt{2}\}$
\Bea
\E[e^{\gamma/\sqrt{c_{n}}}]&=&\int_{0}^{\infty}\Pb\left(\f{\gamma}{\sqrt{c_{n}}}>\log t\right)\;dt\\
&\leq&\int_{0}^{1}\Pb\left(\gamma>\sqrt{c_{n}}\log t\right)\;dt+\int_{1}^{\infty}\Pb\left(\gamma>\sqrt{c_{n}}\log t\right)\;dt\\
&\leq&1+\int_{1}^{\infty}t^{-\sqrt{\kappa c_{n}}/\sqrt{2}}\;dt\\
&=&1+\left(\sqrt{\f{\kappa c_{n}}{2}}-1\right)^{-1}.
\Eea
As a result,
\Bea
\Pb\left(\f{1}{\sqrt{c_{n}}}\|\tilde{X}\|>t\right)\leq ne^{-t}e^{\sqrt{2c_{n}}/\sqrt{\kappa}}.
\Eea
Therefore,
\Bea
\f{1}{c_{n}}\E\|\tilde{X}\|^{2}\leq (\log n+\sqrt{2c_{n}}/\sqrt{\kappa})^{2}\leq Kc_{n}.
\Eea
\end{proof}

\begin{prop}\label{Chapter Band_ESD:Prop: Bound on the difference between trace and quadratic form (with Poincare)}
Let $M$ be one of $C_{j}^{-1},C_{j}^{-1}B_{j}^{-1},C_{j}^{-1}r_{j}r_{j}^{*}C_{j}^{-1*}$ or $C_{j}^{-1}B_{j}^{-1}r_{j}r_{j}^{*}B^{-1*}C_{j}^{-1*}$, and $x_{j}$ be the $j$th column of $X$. In addition, let us also assume that the random variables $x_{ij}$ satisfy the Poincar\'e inequality with constant $m$, and $c_{n}>(\log n)^{2}$. Then we have
\Bea
\E\left|x_{j}^{*}Mx_{j}-\f{c_{n}}{n}\tr M\right|^{2}\leq K c_{n},
\Eea
where $K>0$ is a constant depends on $\Im(z)$, $\sigma$, and the Poincar\'e constant $m$. Moreover, if the entries of the matrix $X$ are bounded by $6\sqrt{\f{2}{\kappa}}\log n$, then 
\Bea
\E\left|x_{j}^{*}Mx_{j}-\f{c_{n}}{n}\tr M\right|^{2l}\leq K c_{n}^{l}(\log n)^{2l},
\Eea 
$K>0$ depends on $l$, $\Im(z)$, $\sigma$, and the Poincar\'e constant $\kappa$.
\end{prop}

\begin{proof}
Let us first prove this for $M=C_{j}^{-1}=(YY^{*}-y_{j}y_{j}^{*}-zI)^{-1}$. Since $x_{ij}$ satisfy the Poincar\'e inequality, they have exponential tails and consequently they have all moments. As a result, we can repeat the same proof of Proposition \ref{Chapter Band_ESD:Prop: Bound on the difference between trace and quadratic form (without Poincare)}. However, notice that in Proposition \ref{Chapter Band_ESD:Prop: Bound on the difference between trace and quadratic form (without Poincare)} we are getting the order $n^{l}$ instead of $c_{n}^{l}$ solely because of the estimate \eqref{Chapter Band_ESD:Eqn: The equation responsible of n^l order in without poincare proposition}. So, it boils down to obtain an estimate of $O(c_{n})$ for \eqref{Chapter Band_ESD:Eqn: The equation responsible of n^l order in without poincare proposition} when $x_{ij}$ satisfy Poincar\'e inequality.

Since $x_{ij}$ satisfy the Poincar\'e inequality we can write
\Bea
\text{Var}\left(\sum_{p\in I_{j}}M_{pp}\right)\leq \f{1}{\kappa}\sum_{s,t}\E\left|\sum_{p\in I_{j}}\f{\partial M_{pp}}{\partial x_{st}}\right|^{2}+\f{1}{\kappa}\sum_{s,t}\E\left|\sum_{p\in I_{j}}\f{\partial M_{pp}}{\partial \bar{x}_{st}}\right|^{2},
\Eea
where $\kappa>0$ is the constant of Poincar\'e inequality. Let $m_{kl}:=\sum_{i\neq j}y_{ki}\bar{y}_{li}=\f{1}{c_{n}}\sum_{i\neq j}(r_{ki}+\s x_{ki})(\bar{r}_{li}+\s\bar{x}_{li})$ be the $kl$th entry of $YY^{*}-y_{j}y_{j}$. It is very easy to compute, and done in the literature in past, that
\Bea
\f{\partial M_{pp}}{\partial m_{kl}}=-\f{1}{1+\d_{kl}}\left[M_{pk}M_{lp}+M_{pl}M_{kp}\right]=-\f{2}{1+\d_{kl}}M_{kp}M_{pl}.
\Eea
Now, it is easy to see that
\Bea
\f{\partial m_{kl}}{\partial \bar x_{st}}=\f{\s}{c_{n}}\sum_{i\neq j}\d_{ks}\d_{it}(r_{li}+\s x_{li})=\f{\s}{c_{n}}\d_{ks}( r_{lt}+\s x_{lt})\1_{\{t\neq j\}}.
\Eea
Consequently,
\Bea
\sum_{p\in I_{j}}\f{\partial M_{pp}}{\partial \bar x_{st}}&=&-\f{\s}{c_{n}}\sum_{p\in I_{j}}\sum_{k,l}\f{2\d_{ks}}{1+\d_{kl}}M_{kp}M_{pl} [r_{lt}+\s x_{lt}]\1_{\{t\neq j\}}\\
&=&-\f{\s}{c_{n}}\sum_{p\in I_{j}}\sum_{l}\f{2}{1+\d_{sl}}M_{sp}M_{pl}[r_{lt}+\s x_{lt}]\1_{\{t\neq j\}}\\
&=&-\f{\s}{c_{n}}\sum_{l}(\tilde{M}_{j})_{sl}[r_{lt}+\s x_{lt}]\1_{\{t\neq j\}}\\
&=&-\f{\s}{\sqrt{c_{n}}}(\tilde{M}_{j}Y_{j})_{st},
\Eea
where $(\tilde{M}_{j})_{sl}=\f{1}{1+\d_{sl}}\sum_{p\in I_{j}}M_{sp}M_{pl}$, and $Y_{j}$ is the matrix $Y$ with $j$th column replaced by zeros.


Let us construct a matrix $(\hat{M_{j}})_{n\times c_{n}}$ from $M$ by removing all the columns except the ones indexed by $I_{j}$. For example, $\hat{M}_{1}$ is the matrix obtained from $M$ by removing $(n-c_{n})$ (i.e., $n-2b_{n}-1$) many columns of $M$ indexed by $b_{n}+2,b_{n}+3,\ldots, n-(b_{n}+1)$. Clearly, $\tilde{M}_{j}=\hat{M}_{j}\hat{M}_{j}^{T}$ (the diagonals are divided by $2$). Therefore, rank$(\tilde{M}_{j})\leq c_{n}$. As a result,
\bea\label{Chapter Band_ESD:eqn: bound of the gradient of partial trace}
\sum_{s,t}\E\left|\sum_{p\in I_{j}}\f{\partial M_{pp}}{\partial \bar x_{st}}\right|^{2}\leq \f{\s^{2}}{c_{n}}\E\text{tr}(\tilde{M}_{j}Y_{j}Y_{j}^{*}\tilde{M}_{j}^{*})\leq\s^{2}\E[\|\tilde{M}_{j}\|^{2}\|Y_{j}Y_{j}^{*}\|]\leq\f{\s^{2}}{|\Im(z)|^{4}}\E[\|Y_{j}Y_{j}^{*}\|],
\eea
where in the last inequality we have used the fact that $\|\hat{M}_{j}\|\leq 1/|\Im(z)|$. Consequently, using the Lemma \ref{Chapter Band_ESD:Lem: Norm of Y is bounded}, we have
\Bea
\sum_{s,t}\E\left|\sum_{p\in I_{j}}\f{\partial M_{pp}}{\partial \bar x_{st}}\right|^{2}\leq Kc_{n}.
\Eea
Repeating the above calculations for $\sum_{s,t}\E\left|\sum_{p\in I_{j}}\f{\partial M_{pp}}{\partial  x_{st}}\right|^{2}$, we can obtain the same bounds. Hence the result follows for $M=C_{j}^{-1}$. 

Since $\|B_{j}^{-1}\|\leq 1/|\Im(z)|$ and $\|r_{j}r_{j}^{*}\|\leq Kc_{n}$, the result follows for $C_{j}^{-1}B_{j}^{-1}$, $C_{j}^{-1}r_{j}r_{j}^{*}C_{j}^{-1*}$, $C_{j}^{-1}B_{j}^{-1}r_{j}r_{j}^{*}B^{-1*}C_{j}^{-1*}$ too.

To prove the second part, we invoke the equation \eqref{Chapter Band_ESD:eqn: Anderson tail bound estimate of Poincare random variables}. 
\Bea
\Pb\left(\left|\sum_{k\in I_{j}}M_{kk}-\E\sum_{k\in I_{j}}M_{kk}\right|>t\right)\leq 2K\exp\left(-\f{\sqrt{\kappa}}{\sqrt{2}\|\|\nabla\sum_{k\in I_{j}}M_{kk}\|_{2}\|_{\infty}}t\right).
\Eea
From the equation \eqref{Chapter Band_ESD:eqn: bound of the gradient of partial trace}, we have
\Bea
\left\|\nabla\sum_{k\in I_{j}}M_{kk}\right\|_{2}^{2}\leq\f{2\s^{2}}{|\Im z|^{4}}\|Y_{j}Y_{j}^{*}\|.
\Eea
Since all the entries of $X$ are bounded by $6\sqrt{\f{2}{\kappa}}\log n$, we have$\|XX^{*}\|\leq K c_{n}^{2}(\log n)^{2}$. And we know that $\|RR^{*}\|\leq Kc_{n}$ for large $n$. Therefore $\|YY^{*}\|\leq Kc_{n}(\log n)^{2}.$ We can get the same bound for $\|Y_{j}Y_{j}^{*}\|$. As a result,
\Bea
\Pb\left(\left|\sum_{k\in I_{j}}M_{kk}-\E\sum_{k\in I_{j}}M_{kk}\right|>t\right)\leq 2K\exp\left(-\f{\sqrt{\kappa}}{{K'\sqrt{2c_{n}}\log n}}t\right).
\Eea
Which implies that 
\Bea
\left|\sum_{k\in I_{j}}M_{kk}-\E\sum_{k\in I_{j}}M_{kk}\right|^{2l}\leq Kc_{n}^{l}(\log n)^{2l}.
\Eea
Plugging this in \eqref{Chapter Band_ESD:Eqn: The equation responsible of n^l order in without poincare proposition}, and following the same procedure as in Proposition \ref{Chapter Band_ESD:Prop: Bound on the difference between trace and quadratic form (without Poincare)}, we have the result.

Observe that the second result of this Proposition is somewhat stronger than the first result, as it leads to the almost sure convergence (see section \ref{Chapter Band_ESD:section: Proof of the theorem (with poincare)}) and it does not need the help of Lemma \ref{Chapter Band_ESD:Lem: Norm of Y is bounded}. However the method used in Lemma \ref{Chapter Band_ESD:Lem: Norm of Y is bounded} is interesting by itself. So we keep it.  
\end{proof}

\section{Appendix}\label{Chapter Band_ESD:section: Appendix}

In this section we list the results which were used in the section \ref{Cahpter Band_ESD:section: main proof of the theorem}. 


\begin{lem}[Lemma 2.3, \cite{silverstein1995empirical}]\label{Chapter Band_ESD:lem: singular value of sum of matrices}
Let $P$, $Q$ be two rectangular matrices of the same size. Then for any $x,y\geq 0$,
\Bea
\mu_{(P+Q)(P+Q)^{*}}(x+y,\infty)\leq\mu_{PP^{*}}(x,\infty)+\mu_{QQ^{*}}(y,\infty).
\Eea
\end{lem}


\begin{lem}[Sherman-Morrison formula]\label{Chapter Band_ESD:Lem:Sherman-Morrison formula} Let $P_{n\times n}$ and $(P+vv^{*})$ be invertible matrices, where $v\in\C^{n}$. Then we have
\Bea
(P+vv^{*})^{-1}=P^{-1}-\f{P^{-1}vv^{*}P^{-1}}{1+v^{*}P^{-1}v}.
\Eea 
In particular,
\Bea
v^{*}(P+vv^{*})^{-1}=\f{v^{*}P^{-1}}{1+v^{*}P^{-1}v}.
\Eea
\end{lem}


\begin{lem}[ Lemma 2.6, \cite{silverstein1995empirical}]\label{Chapter Band_ESD:Lem: Difference between traces of rank one perturbed matrix is bounded}
Let $P$, $Q$ be $n\times n$ matrices such that $Q$ is Hermitian. Then for any $r\in\C^{n}$ and $z=E+i\eta\in\C^{+}$ we have
\Bea
\left|\text{tr}\left((Q-zI)^{-1}-(Q+rr^{*}-zI)^{-1}\right)P\right|=\left|\f{r^{*}(Q-zI)^{-1}P(Q-zI)^{-1}r}{1+r^{*}(Q-zI)^{-1}r}\right|\leq \f{\|P\|}{\eta}.
\Eea
\end{lem}

%

\begin{lem}[\cite{azuma1967weighted}, Lemma 1]\label{Chapter Band_ESD:Lem: Azuma's martingale sum}
Let $\{X_{n}\}_{n}$ be a sequence of random variables such that $|X_{n}|\leq K_{n}$ almost surely, and $\E[X_{i_{1}}X_{i_{2}}\ldots X_{i_{k}}]=0$ for all $k\in \N,\;i_{1}<i_{2}<\cdots<i_{k}$. Then for every $\l\in \R$ we have
\Bea
\E\left[\exp\left\{\l\sum_{i=1}^{n}X_{i}\right\}\right]\leq\exp\left\{\f{\l^{2}}{2}\sum_{i=1}^{n}K_{i}^{2}\right\}.
\Eea

In particular, for any $t>0$ we have
\Bea
\Pb\left(\left|\sum_{i=1}^{n}X_{i}\right|>t\right)\leq 2\exp\left\{-\f{t^{2}}{2\sum_{i=1}^{n}K_{i}^{2}}\right\}.
\Eea
\end{lem}

%
\begin{lem}\label{Chapter Band_ESD:Lem: CDF is bounded by the rank perturbation}
Let $P,Q$ be two $n\times n$ matrices, then 
\Bea
\|\mu_{PP^{*}}-\mu_{QQ^{*}}\|\leq\f{2}{n}rank(P-Q),
\Eea 
where $\|\cdot\|$ denotes the total variation norm between probability measures.
\end{lem}

\begin{proof}
By Cauchy's interlacing property,
\Bea
\|\mu_{PP^{*}}-\mu_{QQ^{*}}\|&\leq&\f{1}{n}\text{rank}(PP^{*}-QQ^{*})\\
&\leq&\f{1}{n}\text{rank}((P-Q)P^{*})+\f{1}{n}\text{rank}(Q(P-Q)^{*})\\
&\leq&\f{2}{n}\text{rank}(P-Q).
\Eea
\end{proof}

\begin{lem}[\cite{bordenave2013localization}, Lemma C.3]\label{Chapter Band_ESD:Lem: Effect of rank one perturbation on the partial trace of resolvent}
Let $P$ and $Q$ be $n\times n$ Hermition matrices, and $I\subset\{1,2,\ldots, n\}$, then
\Bea
\left|\sum_{k\in I}(P-zI)^{-1}_{kk}-\sum_{k\in I}(Q-zI)_{kk}^{-1}\right|\leq\f{2}{\Im(z)}\text{rank}(P-Q).
\Eea
\end{lem}


\begin{lem}\label{Chapter Band_ESD:Lem: Exponential tail bound on the partial trace of the resolvent}
Let $C_{j}$ and $B_{j}$ be defined in \eqref{Chapter Band_ESD:Eqn:Definitions of A, B, C}, $r_{j}$ be the $j$th column of $R$, and $I_{j}\subset\{1,2,\ldots, n\}$ be same as \eqref{Chapter Band_ESD:Def: Definition of Ij}, and $z\in \C^{+}$. Then
\Bea
&&\Pb\left(\left|\sum_{k\in I_{j}}(C_{j}^{-1})_{kk}-\E\sum_{k\in I_{j}}(C_{j}^{-1})_{kk}\right|>t\right)\leq 2\exp\left\{-\f{\Im(z)^{2}t^{2}}{32n}\right\}\\
&&\Pb\left(\left|\sum_{k\in I_{j}}(C_{j}^{-1}B_{j}^{-1})_{kk}-\E\sum_{k\in I_{j}}(C_{j}^{-1}B_{j}^{-1})_{kk}\right|>t\right)\leq 2\exp\left\{-\f{\Im(z)^{2}t^{2}}{32n}\right\}\\
&&\Pb\left(\left|\sum_{k\in I_{j}}(C_{j}^{-1}r_{j}r_{j}^{*}C_{j}^{-1*})_{kk}-\E\sum_{k\in I_{j}}(C_{j}^{-1}r_{j}r_{j}^{*}C_{j}^{-1*})_{kk}\right|>t\right)\leq 2\exp\left\{-\f{\Im(z)^{2}t^{2}}{32n}\right\}\\
&&\Pb\left(\left|\sum_{k\in I_{j}}(C_{j}^{-1}B_{j}^{-1}r_{j}r_{j}^{*}B^{-1*}C_{j}^{-1*})_{kk}-\E\sum_{k\in I_{j}}(C_{j}^{-1}B_{j}^{-1}r_{j}r_{j}^{*}B^{-1*}C_{j}^{-1*})_{kk}\right|>t\right)\leq 2\exp\left\{-\f{\Im(z)^{2}t^{2}}{32n}\right\}.
\Eea
\end{lem}

\begin{proof}
Let $\CF_{l}=\s\{y_{1},\ldots,y_{l}\}$ be the $\s$-algebra generated by the column vectors $y_{1},\ldots, y_{l}$. Then, we can write 
\Bea
&&\sum_{k\in I_{j}}(C_{j}^{-1})_{kk}-\E\sum_{k\in I_{j}}(C_{j}^{-1})_{kk}=\sum_{l=1}^{n}\left[\E\left\{\left.\sum_{k\in I_{j}}(C_{j}^{-1})_{kk}\right|\CF_{l}\right\}-\E\left\{\left.\sum_{k\in I_{j}}(C_{j}^{-1})_{kk}\right|\CF_{l-1}\right\}\right].
\Eea
Notice that for any two matrices $P, Q$, we have $\text{rank}(PP^{*}-QQ^{*})\leq 2\text{rank}(P-Q)$ (from Lemma \ref{Chapter Band_ESD:Lem: CDF is bounded by the rank perturbation}). Therefore, using the Lemma \ref{Chapter Band_ESD:Lem: Effect of rank one perturbation on the partial trace of resolvent} and Lemma \ref{Chapter Band_ESD:Lem: Azuma's martingale sum}, we can conclude the result. The remaining three equations can also be proved in the same way.
\end{proof}

\bibliographystyle{abbrv}
\bibliography{ESD_of_singular_values_of_band.bbl}
\end{document}